\documentclass[a4paper,reqno,10pt]{amsart}
\usepackage{amssymb}
\usepackage{mathrsfs}
\usepackage{mathtools}
\usepackage{amsmath} 
\usepackage{graphicx}
\usepackage{pifont}
\usepackage{slashed}
\usepackage{color}
\usepackage[all]{xy}
\usepackage{type1cm}
\usepackage{hyperref}
\usepackage{color}
\hypersetup{colorlinks=true,
     breaklinks=true,
     linkcolor=blue,
      pageanchor=true}
\newtheorem{thm}{Theorem}[section]
\newtheorem{Prop}[thm]{Proposition}
\newtheorem{Theo}[thm]{Theorem}
\newtheorem{Lem}[thm]{Lemma}
\newtheorem{Koro}[thm]{Corollary}

\theoremstyle{definition}
\newtheorem{Def}[thm]{Definition}
\newtheorem{Bsp}[thm]{Example}

\newcommand{\A}{{\mathcal{A}}}
\newcommand{\B}{{\mathcal B}}
\newcommand{\C}{{\mathcal C}}
\newcommand{\D}{\mathcal D}
\newcommand{\T}{{\mathscr T}}
\newcommand{\X}{{\mathscr X}}
\newcommand{\Y}{{\mathscr Y}}
\newcommand{\Z}{{\mathscr Z}}

\newcommand{\HH}{{\mathcal H}}

\newcommand{\cpx}[1]{#1^{\bullet}}
\renewcommand{\H}{\mathsf{H}}
\newcommand{\Hom}{\mathsf{Hom}}
\newcommand{\RHom}{\mathsf{RHom}}
\newcommand{\Prod}{\mathsf{Prod}}

\title{A note on gluing cosilting objects}

\author[Y. Sun \& Y. Zhang]{Yongliang Sun and Yaohua Zhang*}

\address{\normalfont{Yongliang Sun \\School of Mathematics and Physics, Yancheng Institute of Technology, Jiangsu 224003, China}}
\email{syl13536@126.com}
\address{\normalfont{Yaohua Zhang \\ Hubei Key Laboratory of Applied Mathematics, Faculty of Mathematics and Statistics, Hubei University, Wuhan, 430062, China}}
\email{2160501008@cnu.edu.cn}

\thanks{* Corresponding author.}



\begin{document}
\begin{abstract}
Based on the recent works of M. Saor\'in and A. Zvonoreva on gluing (co)silting objects and of L. Angeler H\"ugel, R. Laking, J. \u{S}\u{t}ov\'i\u{c}ek and J. Vit\'oria on mutating (co)silting objects, we first study further on gluing pure-injective cosilting objects in algebraically compactly generated triangulated categories and gluing cosilting complexes in the derived categories of rings. Then we discuss the compatibility of cosilting gluing and cosilting mutation.
\end{abstract}
\maketitle
\section{Introduction}
Silting objects were introduced by Keller-Vossieck \cite{KV88} to classify $t$-structures and were later investigated by Aihara-Iyama \cite{AI} for studying cluster mutation. Additionally, they have been explored within triangulated categories with coproducts independently by Psaroudakis-Vit\'oria \cite{PV}, and Nicol\'as-Saor\'in-Zvonareva \cite{NSZ}, as well as Angeleri H\"ugel-Marks-Vit\'oria in multiple works \cite{A19, AMV16, AMV17, AMV20}. It has been demonstrated that silting objects are intricately related to $t$-structures and co-$t$-structures, which are foundational concepts in the representation theory of Artin algebras \cite{KY, MSSS, W}. There are two basic operations on these objects: silting gluing and silting mutation. The gluing techniques, inspired by the recollement theory of Beilinson, Bernstein, and Deligne \cite{BBD}, have been extensively studied in algebra representation theory and homological algebra. Silting gluing has been explored by Aihara and Iyama in Hom-finite Krull-Schmidt triangulated categories \cite{AI}, more explicitly by Liu, Vit\'oria, and Yang in Krull-Schmidt triangulated categories \cite{LVY}, and by Saor\'in and Zvonareva \cite{SZ} in triangulated categories with coproducts. These studies provide methods to construct silting objects in the central category of a recollement from those in the outer categories.
Silting mutation has been examined by Aihara and Iyama \cite{AI} and Buan, Reiten, and Thomas \cite{BRT} in Hom-finite Krull-Schmidt triangulated categories. Recently, Angeleri H\"ugel, Laking, \u{S}\u{t}ov\'i\u{c}ek, and Vit\'oria have developed a framework to investigate silting or cosilting mutation in triangulated categories with coproducts \cite{ALSV22}.

As the duals of silting objects, cosilting objects have also garnered considerable attention and have been extensively studied by various authors, including \cite{ALSV22, PV, ZW17}. Cosilting gluing was explored in \cite{SZ}, while cosilting mutation was investigated in \cite{ALSV22}.

Regarding these two operations, a natural question arises, which serves as the primary direct motivation for this paper.

\vspace{1em}
{\bf Question:} Are (co)silting mutation and (co)silting gluing in triangulated categories compatible?
\vspace{1em}

The first significant result related to this question is attributed to Aihara and Iyama \cite[Lemma 2.40]{AI} (or see \cite[Corollary 6.9]{LVY} for a more explicit presentation). In \cite[Corollary 6.9]{LVY}, it was demonstrated that silting mutation and silting gluing in a Krull-Schmidt triangulated category are compatible. Beyond this, we have not found any further discussion on this question. In this paper, we provide a partial answer in the context of cosilting theory, specifically within triangulated categories with products. We show that cosilting mutation and cosilting gluing in such categories are compatible under certain conditions.

To achieve our goal, we initially build upon the concepts of gluing cosilting objects as developed by Saorín and Zvonareva \cite{SZ}. We then advance the discussion by examining the gluing of pure-injective cosilting objects within a ladder of algebraically compactly generated triangulated categories of height 2 (see Theorem~\ref{thm:glue pure-inj cosilting}), as well as the gluing of cosilting complexes within a ladder of derived categories of rings of height 2 (see Theorem~\ref{thm:glue cosilting complex}). With these foundations established, we explore the compatibility of pure-injective cosilting mutation and pure-injective cosilting gluing within a recollement of triangulated categories (see Theorem~\ref{thm:muta at left} and Theorem~\ref{thm:muta at right}) under certain conditions. It is important to note that our methods can be transferred dually to the context of silting mutation and silting gluing in triangulated categories with coproducts.

The contents of this paper are organized as follows. In Section~\ref{sec:preliminaries}, we establish some notation and recall the definitions and facts that are pertinent to our discussion. In Section~\ref{sec:glue objects}, we revisit Saorín-Zvonareva's method of gluing cosilting objects and present a sufficient condition that enables this gluing process. This condition is also applied to the gluing of pure-injective cosilting objects in algebraically compactly generated triangulated categories. To support our objectives, we utilize several characterizations of pure-injective cosilting objects. In Section 4, we focus on gluing cosilting complexes within derived categories, which relies heavily on a specific characterization of cosilting complexes. Finally, in Section~\ref{sec:mutation}, we provide partial answers to the aforementioned question.

\section{Preliminaries}\label{sec:preliminaries}
Let $\A$ be an additive category, with the Hom-spaces in $\A$ denoted simply by $\A(-, -)$ simply. Let $\mathcal{S}$ be a class of objects of $\A$, by $\mathcal{S}{^{\perp}}$(resp., $^{\perp}\mathcal{S}$) we denote the full subcategory consisting of objects $A\in \A$ such that $\A(S,A)=0$(resp., $\A(A,S)=0)$, for all $S\in\mathcal{S}$. Furthermore, if $\A$ is a triangulated category with a shift functor $[1]$, for an integer $n\in \mathbb{N}$, we define
$$\mathcal{S}^{\perp_{\geq n}}:=\{X\in \A~|~{\A}(S,X[k])=0,~\forall~S\in\mathcal{S},~k\geq n\}$$ 
$${}^{\perp_{\geq n}}\mathcal{S}:=\{X\in \A~|~{\A}(X,S[k])=0,~\forall~S\in\mathcal{S},~k\geq n\}$$ 
Similarly, we can define $\mathcal{S}^{\perp_{\leq n}}$ and $^{\perp_{\leq n}}\mathcal{S}$.

\subsection{Torsion pairs, $t$-structures and co-$t$-structures}
\begin{Def}
Let $\T$ be a triangulated category, a pair of subcategories $(\mathcal{U},\mathcal{V})$ in $\T$ is a {\em torsion pair} (\cite{IY}) if
\begin{enumerate}
    \item $\mathcal{U}$ and $\mathcal{V}$ are closed under direct summands;
    \item ${\T}(\mathcal{U},\mathcal{V})=0$;
    \item for each $T\in\T$, there is a triangle $U\to T\to V\to U[1]$ in $\T$ with  $U\in\mathcal{U}$ and $V\in \mathcal{V}$.
\end{enumerate}

{\color{black}A pair $(\T^{\leq 0},\T^{\geq 0})$ of $\T$ is called a {\em $t$-structure} (\cite{BBD}) if $(\T^{\leq 0},\T^{\geq 0}[-1])$ is a torsion pair} and $\T^{\leq 0}[1]\subseteq \T^{\leq 0}$. In this case, the left (resp., right) part of the $t$-structure is called an {\em aisle} (resp., {\em coaisle}). The $t$-structure is {\em nondegenerate} if 
$$\underset{n\in\mathbb{Z}}{\bigcap}\T^{\leq 0}[n]=0=\underset{n\in\mathbb{Z}}{\bigcap} \T^{\geq 0}[n].$$
The {\em heart} of the $t$-structure is defined as $\mathcal{H}=\T^{\leq 0}\cap \T^{\geq 0}$.
{\color{black}A pair $(\T_{\geq 0},\T_{\leq 0})$ of $\T$ is called a {\em co-$t$-structure} (\cite{Bondarko}) if $(\T_{\geq 0}[-1],\T_{\leq 0})$ is a torsion pair and 
$\T_{\geq 0}[-1]\subseteq \T_{\geq 0}$.}
\end{Def}

Let $(\mathcal{U}, \mathcal{V})$ and $(\mathcal{V}, \mathcal{W})$ be torsion pairs, we say  $(\mathcal{U}, \mathcal{V})$ is {\em left adjacent} to $(\mathcal{V}, \mathcal{W})$ and $(\mathcal{V}, \mathcal{W})$ is {\em right adjacent} to $(\mathcal{U}, \mathcal{V})$. In this case, the triple $(\mathcal{U}, \mathcal{V}, \mathcal{W})$ is called a {\em torsion-torsion-free(TTF) triple}.

Let $A$ be a ring and $(\D^{\leq 0}, \D^{\geq 0})$ be the standard $t$-structure in $\D(A)$, define $\D^{\leq n}:=\D^{\leq 0}[-n]$ and $\D^{\geq n}:=\D^{\geq 0}[-n]$. This $t$-structure has both a left and a right adjacent co-$t$-structures. Indeed, let $K_{\leq 0}$ stand for the subcategory of objects in $\D(A)$ which are isomorphic to a complex $\cpx{Z}$ of injective $A$-modules such that $Z^{i}=0$ for all $i>0$, $(\D^{\geq 0}, K_{\leq 0})$ is the right adjacent co-$t$-structure of $(\D^{\leq 0}, \D^{\geq 0})$ (see \cite[Example 2.4]{AMV17}).

\subsection{Silting and cosilting objects}
\begin{Def}(\cite[Definition 4.1]{PV}, \cite[Definition 6.1]{A19})
An object $M$ in a triangulated category $\T$ is called:
\begin{enumerate}
    \item {\em silting} if $(M^{\perp_{>0}},M^{\perp_{<0}})$ is a $t$-structure in $\T$;
    \item {\em cosilting} if $({^{\perp_{<0}}M}, {^{\perp_{>0}}M})$ is a $t$-structure in $\T$.
\end{enumerate}
\end{Def}
Let $M$ be a cosilting object in $\T$, then $M\in {^{\perp_{>0}}M}$ (i.e. self-orthogonal in positive degrees) and generates $\T$(i.e. $X=0$ provided $\T(X, M[i])=0$ for $i\in \mathbb{Z}$). 

Two cosilting objects $M$ and $N$ are called {\em equivalent} if $\mathsf{Prod}(M)=\mathsf{Prod}(N)$, where the notation $\mathsf{Prod}(X)$ denotes the subcategory whose objects are summands of existing
products of copies of $X$. According to \cite[Proposition 4.3, Lemma 4.5]{PV}, the definition is equivalent to that in \cite[Definition 2.5]{AMV17}, namely, the associated $t$-structures of $M$ and $N$ are the same. 
It is well-known that $M$ is silting in $\T$ if and only if $M$ is cosilting in $\T^{\mathsf{op}}$.

A triangle $X\stackrel{f}{\to} Y\stackrel{g}{\to} Z\stackrel{h}{\to} X[1]$ in $\T$ is a {\em pure triangle} (\cite{K00}) if $\forall ~C\in\T^{c}$ the sequence 
$$0\longrightarrow \T(C, X)\stackrel{f^*}\longrightarrow \T(C,Y)\stackrel{g^*}\longrightarrow \T(C,Z)\longrightarrow 0$$ is exact. An object $E\in\T$ is {\em pure-injective} (\cite{K00}) if for any pure triangle $X\stackrel{f}{\to} Y\stackrel{g}{\to} Z\stackrel{h}{\to} X[1]$, the induced sequence 
$$0\longrightarrow \T(Z,E)\stackrel{g_*}\longrightarrow \T(Y,E)\stackrel{f_*}\longrightarrow \T(X,E)\longrightarrow 0$$
is exact.

\subsection{Recollements and ladders}
\begin{Def}
Let $\mathscr{T}$, $\X$ and $\Y$ be triangulated categories. The category $\mathscr{T}$ is a {\em recollement} (\cite{BBD}) of $\X$
and $\Y$ if there is a diagram of six triangulated functors 
$$\xymatrix{\X\ar^-{i_*=i_!}[r]&\mathscr{T}\ar^-{j^!=j^*}[r]
\ar^-{i^!}@/^1.2pc/[l]\ar_-{i^*}@/_1.6pc/[l]
&\Y\ar^-{j_*}@/^1.2pc/[l]\ar_-{j_!}@/_1.6pc/[l]}$$ such
that
\begin{enumerate}
    \item $(i^*,i_*),(i_!,i^!),(j_!,j^!)$ and $(j^*,j_*)$ are adjoint
pairs;
    \item $i_*,j_*$ and $j_!$ are fully faithful functors;
    \item $i^!j_*=0$; and
    \item for each object $T\in\mathscr{T}$, there are two triangles in
$\T$
$$
i_!i^!(T)\longrightarrow T\longrightarrow j_*j^*(T)\longrightarrow i_!i^!(T)[1],
$$
$$
j_!j^!(T)\longrightarrow T\longrightarrow i_*i^*(T)\longrightarrow j_!j^!(T)[1].
$$
\end{enumerate}
\end{Def}

 Let $(\X^{\leq 0},\X^{\geq 0})$ and $(\Y^{\leq 0}, \Y^{\geq 0})$ be $t$-structures in $\X$ and $\Y$, respectively. The {\em glued $t$-structure} (\cite{BBD}) $(\T^{\leq 0},\T^{\geq 0})$ of $\T$ along the recollement is defined as 
$$\T^{\leq 0}:=\{T\in\T| i^{*}(T)\in \X^{\leq 0}, j^{*}(T)\in \Y^{\leq 0}\},$$
$$\T^{\geq 0}:=\{T\in\T| i^{!}(T)\in \X^{\geq 0}, j^{*}(T)\in \Y^{\geq 0}\}.$$
Obviously, $i^{*}(\T^{\leq 0})=\X^{\leq 0}$, $j^{*}(\T^{\leq 0})=\Y^{\leq 0}$, $i^{!}(\T^{\geq 0})=\X^{\geq 0}$ and $j^{*}(\T^{\geq 0})=\Y^{\geq 0}$.
\begin{Def}
A {\em ladder} (\cite{AKLY}) is a finite or infinite diagram of triangulated categories and triangulated functors
$$\xymatrix{
&
\vdots
&&
\vdots
&\\
\X
\ar@<1.5ex>[rr]|-{i_{n-1}}
\ar@<-1.5ex>[rr]|-{i_{n+1}}
&&
\ar@<3ex>[ll]|-{j_{n-2}}
\ar[ll]|-{j_{n}}
\ar@<-3ex>[ll]|-{j_{n+2}}
\T
\ar@<1.5ex>[rr]|-{j_{n-1}}
\ar@<-1.5ex>[rr]|-{j_{n+1}}
&&
\ar@<3ex>[ll]|-{i_{n-2}}
\ar[ll]|-{i_{n}}
\ar@<-3ex>[ll]|-{i_{n+2}}
\Y\\
&
\vdots
&&
\vdots
&}
$$
such that any three consecutive rows form a recollement. The {\em height} of a ladder is the number of recollements contained in it (counted with multiplicities).
\end{Def}

\section{Gluing cosilting objects}\label{sec:glue objects}
In this section, we revisit a method of gluing cosilting objects developed by Saor\'in and Zvonareva. We apply this technique to glue pure-injective cosilting objects within algebraically compactly generated triangulated categories. 
\begin{Def}\label{def:gluable cosilting}
    Let
$$\xymatrix{\mathscr{T}_1\ar^-{i_*=i_!}[r]&\mathscr{T}\ar^-{j^!=j^*}[r]
\ar^-{i^!}@/^1.2pc/[l]\ar_-{i^*}@/_1.6pc/[l]
&\mathscr{T}_2\ar^-{j_*}@/^1.2pc/[l]\ar_-{j_!}@/_1.6pc/[l]}$$
be a recollement of triangulated categories, $C_1$ and $C_2$ be cosilting objects in $\mathscr{T}_1$ and $\mathscr{T}_2$, respectively. The pair $(C_1,C_2)$ is called a {\em gluable cosilting pair} if $i_{*}(C_1)$ admits a triangle
$$(*)~~~~~V\longrightarrow i_{*}(C_1)\longrightarrow U\longrightarrow V[1]~\text{with}~V\in j_{*}(^{\perp_{\geq 0}}C_2), U\in (j_{*}(^{\perp_{\geq 0}}C_2))^{\perp}.$$
\end{Def}

\begin{Theo}\label{thm:glue cosilting}\textnormal{(\cite[Theorem 6.3 and Remark 6.7]{SZ})}
If $(C_1, C_2)$ is a gluable cosilting pair, then $C:= j_{*}(C_2)\oplus U$ is a cosilting object in $\T$ such that its associative $t$-structure $(^{\perp_{<0}}C, {^{\perp_{>0}}C})$ is the glued $t$-structure from $(^{\perp_{< 0}}C_2, {^{\perp_{> 0}}C_2})$ and  $(^{\perp_{< 0}}C_1, {^{\perp_{> 0}}C_1})$. That means $C$ is independent of the choice of $U$ up to equivalence.
\end{Theo}

The following proposition is devoted to understanding the triangle $(*)$ in the definition above.

\begin{Prop}\label{rem:from condition tri}
   \begin{enumerate}
    \item $C_1\simeq i^!(U)$;
    \item $\T(V, U[k])=0, \forall~k\geq 0$. Furthermore, $V[1]\in {^{\perp_{>0}}C}$;
    \item The triangle $(*)$ in Definition~\ref{def:gluable cosilting} is isomorphic to the canonical triangle
    $$j_*j^*(U)[-1]\longrightarrow i_*i^!(U)\longrightarrow U\longrightarrow j_*j^*(U).$$
Moreover, the existence of the triangle $(*)$ is equivalent to the existence of an object $U\in (j_{*}(^{\perp_{\geq 0}}C_2))^{\perp}$ such that $i^!(U)\simeq C_1$ and $j^*(U)\in ^{\perp_{>0}}C_2$.
\end{enumerate} 
\end{Prop}
\begin{proof}
(1) and (2) are left to the reader. We will prove (3). Since both the triangle $(*)$ and the canonical triangle are  decomposition triangles of $U$ with respect to the $t$-structure $(i_*(\T_1), j_*(\T_2))$ in $\T$, so they are isomorphic. Given such an object $U$ as described in $(3)$, the canonical triangle of $U$ is precisely a triangle for $i_{*}(C_1)$ that satisfies the condition $(*)$.
\end{proof}

The proposition below is a dual of \cite[Corollary 6.5]{SZ} and presents a sufficient condition to satisfy the condition $(*)$. For the convenience of the reader, a detailed proof is included.

\begin{Prop}\label{prop:sufficient glue}\textnormal{(Dual of \cite[Corollary 6.5]{SZ})}
Let the following diagram be a ladder of triangulated categories of height 2
$$(\#)\quad\xymatrix{
\mathscr{T}_1
\ar@<-2.25ex>[rr]|-{i_{\#}}
\ar@<0.75ex>[rr]|-{i_{*}}
&&
\ar@<-2.25ex>[ll]|-{i^{*}}
\ar@<0.75ex>[ll]|-{i^{!}}
\T
\ar@<-2.25ex>[rr]|-{j^{\#}}
\ar@<0.75ex>[rr]|-{j^{*}}
&&
\ar@<-2.25ex>[ll]|-{j_{!}}
\ar@<0.75ex>[ll]|-{j_{*}}
\mathscr{T}_2}.$$
Assume $C_1$ and $C_2$ are cosilting objects in $\mathscr{T}_1$ and $\mathscr{T}_2$ respectively and the $t$-structure $({^{\perp_{< 0}}C_2}, {^{\perp_{> 0}}C_2})$ has an adjacent co-$t$-structure $({^{\perp_{> 0}}C_2}, (^{\perp_{\geq 0}}C_2)^{\perp})$. Then $(C_1,~C_2)$ is a gluable cosilting pair concerning the upper recollement.
\end{Prop}
\begin{proof}
We have the following diagram {\color{black}with objects in $\T$} by the octahedral axiom of triangulated categories
 $$
(\#\#)\quad \xymatrix{
j_{*}(V)\ar[d]\ar@{=}[r]&
 j_{*}(V)\ar@{-->}[d]\\
 j_{*}j^{\#}i_{*}(C_1)\ar[d]\ar[r]&
 i_{*}(C_1)\ar@{-->}[d]\ar[r]&
 i_{\#}i^{!}i_{*}(C_1)\ar@{=}[d]\ar[r]&
 j_{*}j^{\#}i_{*}(C_1)[1]\ar[d]\\
 j_{*}(U')\ar[d]\ar[r]&
 U\ar@{-->}[d]\ar[r]&
 i_{\#}i^{!}i_{*}(C_1)\ar[r]&
 j_{*}(U')[1]\\
 j_{*}(V)[1]\ar@{=}[r]&
 j_{*}(V)[1]
 }$$
where the first column is the image of $j_*$ of a decomposition triangle of $j^\#i_*(C_1)$ with respect to the co-$t$-structure $(^{\perp_{> 0}}C_2, (^{\perp_{\geq 0}}C_2)^{\perp})$ and the second row is the canonical triangle of $i_*(C_1)$ from the definition of recollement. {\color{black}In the diagram, $j_*(V)\in j_*(^{\perp_{\geq 0}}C_2)$ and $j_*(U')\in j_*((^{\perp_{\geq 0}}C_2)^{\perp}) \subseteq (j_*(^{\perp_{\geq 0}}C_2))^{\perp}$ ($j_*$ is fully-faithful). Since $ i_{\#}i^{!}i_{*}(C_1)\in (j_*(^{\perp_{\geq 0}}C_2))^{\perp}$, then the third-row triangle tells us $U\in (j_*(^{\perp_{\geq 0}}C_2))^{\perp}$.} The triangle in the second column 
$$j_*(V)\longrightarrow i_*(C_1)\longrightarrow U\longrightarrow j_*(V)[1]$$
satisfies the condition $(*)$ in Definition~\ref{def:gluable cosilting}. We finish the proof.
\end{proof}

Let $\T$ be a triangulated category. A subcategory $\mathcal{V}$ of $\T$ is called {\em cosuspended} (\cite[Defnition 3.2]{MV18}) if $\mathcal{V}$ is closed under extensions and $\mathcal{V}[-1]\subseteq \mathcal{V}$. Assuming $\T$ has coproducts, we denote $\T^{c}$ the subcategory of all compact objects in $\T$. A subcategory $\mathcal{C}$ of $\T$ is called {\em definable} (\cite[Definition 4.1]{AMV17} or \cite[Definition 3.9]{MV18}) if there is a set $S$ of morphisms in $\T^{c}$ such that 
$$\mathcal{C}=\{X\in\T\mid {\T}(f,X)~\text{is surjective for all}~f\in S\}.$$

A TTF triple $(\X, \Y, \Z)$ of subcategories of $\T$ is called {\em cosuspended} if, in addition, $\Y$ is cosuspended, this is equivalent to say $(\X, \Y)$ is a $t$-structure and $(\Y, \Z)$ is a co-$t$-structure. In this case, $\Y$ is called a {\em cosuspended TTF class}. Furthermore, if $(\X, \Y)$ is nondegenerate, then we call $\Y$ a {\em nondegenerate cosuspended TTF class}. The lemma below is a result of gluing definable subcategories with this property.

\begin{Lem}\label{lem:glue ttf}\textnormal{(\cite[Theorem 5.9]{PV19})}
Let $(\T_1, \T, \T_2)$ be a ladder of compactly generated triangulated categories of height 2, as described in $(\#)$ (see Proposition~\ref{prop:sufficient glue}). Let $\D_1$ and $\D_2$ be definable subcategories of $\mathscr{T}_1$ and $\mathscr{T}_2$, respectively. We define the glued subcategory $\D$ of $\T$ along the upper recollement as follows:
$$\D:=\{T\in\T| i^{!}(T)\in \D_1, j^{*}(T)\in \D_2\}$$
\begin{enumerate}
    \item $\D_1$ and $\D_2$ are definable, then so is $\D$.
    \item If both $\D_1$ and $\D_2$ are nondegenerate cosuspended TTF classes, then so is the subcategory $\D$.   
\end{enumerate}
\end{Lem}
\begin{proof}
    The assumption that ``$i_*$ preserves compact objects" in \cite[Theorem 5.9]{PV19} is equivalent to the statement that the recollement can extend one step downwards, i.e., it can be completed to a ladder of height of 2 {\color{black}(\cite[Theorem 4.1]{N})}. Therefore, (1) follows from the proof of \cite[Theorem 5.9]{PV19} and (2) is derived from \cite[Theorem 5.9(1)]{PV19}.
\end{proof}

Now, let us move to the case of cosilting objects. 

\begin{Lem}\label{lem:definable}\textnormal{(\cite[Lemma 4.8]{AMV17})}
Let $\T$ be a compactly generated triangulated category and $C$ a cosilting object in $\T$. Then 
$C$ is pure-injective if and only if ${^{\perp_{>0}}C}$ is definable.
\end{Lem}

The following theorem is proved implicitly in \cite[Theorem 5.8, Theorem 5.9]{PV19}. For the convenience of the reader, we include a proof.

\begin{Theo}\label{thm:glue pure-inj cosilting}
Let the following diagram be a ladder of algebraically compactly generated triangulated categories of height 2:
$$\xymatrix{
\mathscr{T}_1
\ar@<-2.25ex>[rr]|-{i_{\#}}
\ar@<0.75ex>[rr]|-{i_{*}}
&&
\ar@<-2.25ex>[ll]|-{i^{*}}
\ar@<0.75ex>[ll]|-{i^{!}}
\T
\ar@<-2.25ex>[rr]|-{j^{\#}}
\ar@<0.75ex>[rr]|-{j^{*}}
&&
\ar@<-2.25ex>[ll]|-{j_{!}}
\ar@<0.75ex>[ll]|-{j_{*}}
\mathscr{T}_2}$$
Let $C_1$ and $C_2$ be pure-injective cosilting objects in $\mathscr{T}_1$ and $\mathscr{T}_2$, respectively. Then $j_{*}(C_2)\oplus U$ is also a pure-injective cosilting object.
\end{Theo}

\begin{proof}
By \cite[Theorem 6.13]{A19}, the t-structures $({^{\perp_{<0}}C_1},{^{\perp_{>0}}C_1})$ and $({^{\perp_{<0}}C_2},{^{\perp_{>0}}C_2})$ have adjacent co-$t$-structures. It follows from Proposition~\ref{prop:sufficient glue} that $j_{*}(C_2)\oplus U$ is a cosilting object. Since $^{\perp_{>0}}(j_{*}(C_2)\oplus U)$ is glued from ${^{\perp_{>0}}C_1}$ and ${^{\perp_{>0}}C_2}$, both of which are definable, it follows from Lemma~\ref{lem:glue ttf}(1) that $^{\perp_{>0}}(j_{*}(C_2)\oplus U)$ is definable. Thus $j_{*}(C_2)\oplus U$ is pure-injective by Lemma~\ref{lem:definable}.
\end{proof}

The last theorem informs us that in a ladder of algebraically compactly generated triangulated categories of height 2, a pair of pure-injective cosilting objects is always gluable along the upper recollement.

\section{Gluing cosilting complexes}\label{sec:glue complexes}
In the context of derived categories of rings, it is important to acknowledge cosilting complexes. Generally, cosilting complexes are pure-injective cosilting objects (\cite[Proposition 3.10]{MV18}).

\begin{Def}
Let $A$ be a ring. An object $\cpx{X}$ in $\D(A)$ is a {\em cosilting complex} (\cite{ZW17}) if satisfies
\begin{enumerate}
    \item ${\D(A)}({\cpx{X}}^{J}, \cpx{X}[i])=0$ for $i>0$ and any set $J$, and
    \item thick($\Prod(\cpx{X}$))=$K^{b}(A\mbox{-}\mathsf{Inj})$.
\end{enumerate}
\end{Def}

Let $A$ be a ring and $(\D^{\leq 0}, \D^{\geq 0})$ be the standard $t$-structure in $\D(A)$, a cosuspended subcategory $\mathcal{V}$ of $\D(A)$ is called {\em cointermediate} (\cite[Definition 6.14]{A19}) if there are integers $n\leq m$ such that $\D^{\geq m}\subseteq \mathcal{V}\subseteq \D^{\geq n}$.

\begin{Lem}\label{lem:glue cointermediate}
Let 
$$\xymatrix{
\D(A)\ar@<-2.25ex>[rr]|-{i_{\#}}\ar@<0.75ex>[rr]|-{i_{*}}
&&
\ar@<-2.25ex>[ll]|-{i^{*}}
\ar@<0.75ex>[ll]|-{i^{!}}
\D(B)
\ar@<-2.25ex>[rr]|-{j^{\#}}
\ar@<0.75ex>[rr]|-{j^{*}}
&&
\ar@<-2.25ex>[ll]|-{j_{!}}
\ar@<0.75ex>[ll]|-{j_{*}}
\D(C)}$$
be a ladder of derived categories of height 2. Let $\Y_C$ and $\Y_A$ be subcategories of $\D(C)$ and $\D(A)$, respectively. Define a subcategory of $\D(B)$
$$\Y_B:=\{Y\in \D(B)~|~j^*(Y)\in \Y_C,i^!(Y)\in \Y_A\}.$$
Then $\Y_C$ and $\Y_A$ are cointermediate if and only if so is $\Y_B$.
\end{Lem}

\begin{proof}
    Set
$$K^{[a,b]}(A\mbox{-}\mathsf{proj}):=\{\cpx{P}|\cpx{P}\in K^{b}(A\mbox{-}\mathsf{proj})~\text{with}~ P^{i}=0,~i\notin [a,b]\}, a\leq b\in\mathbb{Z}.$$ 
Since $i^{*}, i_{*}, j_{!}, j^{*}$ preserve compact objects (\cite[Theorem 5.1]{N}), we can assume 
$$i^{*}(B)\in K^{[t_{1},t_{2}]}(A\mbox{-}\mathsf{proj}),~i_{*}(A)\in K^{[s_{1},s_{2}]}(B\mbox{-}\mathsf{proj}),$$
$$ j_{!}(C)\in K^{[u_{1},u_{2}]}(B\mbox{-}\mathsf{proj}),~j^{*}(B)\in K^{[v_{1},v_{2}]}(C\mbox{-}\mathsf{proj}).$$ 

$(\Rightarrow)$ 
By our assumption, there are integers $m_1>n_1$, $m_2>n_2$ such that
$$\D^{\geq m_{1}}(A)\subseteq \Y_A\subseteq \D^{\geq n_{1}}(A),~\D^{\geq m_{2}}(C)\subseteq \Y_C\subseteq \D^{\geq n_{2}}(C).$$ 
Denote $m=\text{max}\{m_{1}+s_{2}, m_{2}+u_{2}\}$, $n=\text{min}\{n_{1}-t_{2}, n_{2}-v_{2}\}$, 
so it suffices to prove $\D^{\geq m}(B)\subseteq \Y_{B}\subseteq \D^{\geq n}(B)$.

Let $\cpx{Z}\in \D^{\geq m}(B)$, then 
$${\D(A)}(A, i^{!}(\cpx{Z})[i])\simeq {\D(B)}(i_{*}(A), \cpx{Z}[i])\simeq 0, i\leq m_{1}-1,$$
$${\D(C)}(C,j^{*}(\cpx{Z})[i])\simeq {\D(B)}(j_{!}(C), \cpx{Z}[i])\simeq 0, i\leq m_{2}-1,$$
these imply that
$$i^{!}(\cpx{Z})\in \D^{\geq m_{1}}(A)\subseteq\Y_A~\text{and}~ j^{*}(\cpx{Z})\in \D^{\geq m_{2}}(C)\subseteq \Y_C.$$
Hence, we have $\cpx{Z}\in \Y_B$.

Let $\cpx{Z}\in \Y_B$, then
$$ j^*(\cpx{Z})\in \Y_C\subseteq \D^{\geq n_{2}}(C)~\text{and}~i^!(\cpx{Z})\in \Y_A\subseteq \D^{\geq n_{1}}(A)$$
Consider the triangle
$$j_{!}j^{*}(B)\longrightarrow B\longrightarrow i_{*}i^{*}(B)\longrightarrow j^{*}j_{!}(B)[1],$$
since for $i\leq n-1$, there are
$${\D(B)}(i_{*}i^{*}(B),\cpx{Z}[i])\simeq {\D(A)}(i^{*}(B),i^{!}(\cpx{Z})[i])\simeq 0,$$ 
$${\D(B)}(j_{!}j^{*}(B),\cpx{Z}[i])\simeq {\D(C)}(j^{*}(B),j^{*}(\cpx{Z})[i])\simeq 0,$$
then we have ${\D(B)}(B, \cpx{Z}[i])\simeq 0, i\leq n-1$, thus $\cpx{Z}\in \D^{\geq n}(B)$. 

In conclusion, we obtain $\D^{\geq m}(B)\subseteq \Y_B\subseteq \D^{\geq n}(B)$ which means $\Y_B$ is cointermediate.

$(\Leftarrow)$ Since $\Y_C=j^*j_*(\Y_C)\subseteq j^*(\Y_B)\subseteq\Y_C$ and $\Y_A=i^!i_*(\Y_A)\subseteq i^!(\Y_B)\subseteq \Y_A$, then $\Y_C=j^*(\Y_B)$ and $\Y_A=i^!(\Y_B)$, so 
it suffices to show $j^*(\Y_{B})$ and $i^!(\Y_B)$ are cointermediate subcategories in $\D(C)$  and $\D(A)$ respectively. Under our condition, we can assume $\D^{\geq m}(B)\subseteq \Y_{B}\subseteq \D^{\geq n}(B)$ for integers $n<m$.

We first prove $j^*(\Y_{B})$ is cointermediate in $\D(C)$.  Since 
$$j_!(C)[u_{2}]\in K^{[u_{1}-u_{2},0]}(B\mbox{-}\mathsf{proj})~\text{and}~\Y_B\subseteq \D^{\geq n}(B),$$
then for $i\leq 0$,
\begin{align*}
    {\D(C)}(C, j^{*}(\Y_B)[n-u_{2}-1][i])&\simeq {\D(C)}(C, j^{*}(\Y_B)[n-u_{2}+i-1])\\
    &\simeq {\D(B)}(j_!(C)[u_{2}], \Y_{B}[n+i-1])\\
    &= 0
\end{align*}
this implies $j^{*}(\Y_B)\subseteq \D^{\geq n-u_2-1}(C)$. Note that for $i\leq -m+1$, we have 
$${\D(B)}(B,j_*(\D^{\geq v_2-m+2}(C))[i])\simeq {\D(C)}(j^*(B), \D^{\geq v_2-m+2}(C)[i])= 0,$$
then $j_*(\D^{\geq v_2-m+2}(C))\subseteq \D^{\geq m}(B)\subseteq \Y_{B}$. Hence, $\D^{\geq v_2-m+2}(C)\subseteq j^*(\Y_B)$. Therefore, $j^*(\Y_{B})$ is cointermediate in $\D(C)$.

By a similar proof as above, one can show that 
$$\D^{\geq t_{2}-m+2}(A)\subseteq i^!(\Y_B)\subseteq \D^{\geq n-s_2-1}(A).$$
This implies that $i^!(\Y_B)$ is cointermediate in $\D(A)$.
\end{proof}

\begin{Theo}\label{thm:glue cosilting complex}
Let 
$$\xymatrix{
\D(A)\ar@<-2.25ex>[rr]|-{i_{\#}}\ar@<0.75ex>[rr]|-{i_{*}}
&&
\ar@<-2.25ex>[ll]|-{i^{*}}
\ar@<0.75ex>[ll]|-{i^{!}}
\D(B)
\ar@<-2.25ex>[rr]|-{j^{\#}}
\ar@<0.75ex>[rr]|-{j^{*}}
&&
\ar@<-2.25ex>[ll]|-{j_{!}}
\ar@<0.75ex>[ll]|-{j_{*}}
\D(C)}$$
be a ladder of derived categories of height 2.
If $\cpx{X}$ and $\cpx{Y}$ are cosilting complexes in $\D(A)$ and $\D(C)$, respectively, then so is $j_{*}(\cpx{Y})\oplus U$ in $\D(B)$.
\end{Theo}

\begin{proof}
By \cite[Theorem 3.13]{MV18}, the $t$-structure $(^{\perp_{<0}}\cpx{Y}, {^{\perp_{>0}}\cpx{Y}})$ has a right adjacent co-$t$-structure, then $j_{*}(\cpx{Y})\oplus U$ is a cosilting object in $\D(B)$ by Proposition~\ref{prop:sufficient glue}. It follows from Lemma~\ref{lem:glue ttf} and Lemma~\ref{lem:glue cointermediate} that the coaisle of the associated $t$-structure of $j_{*}(\cpx{Y})\oplus U$ (see Theorem~\ref{thm:glue cosilting})
 $$\mathcal{V}=\{\cpx{U}\in \D(B)| i^{!}(\cpx{U})\in {^{\perp_{>0}}\cpx{X}}, j^{*}(\cpx{U})\in {^{\perp_{>0}}\cpx{Y}}\}$$
 is co-suspended, cointermediate, and definable. Hence, $j_{*}(\cpx{Y})\oplus U$ is a cosilting complex according to \cite[Theorem 3.14]{MV18}. We finish the proof.
\end{proof}

\begin{Bsp}
Let $A$ and $B$ be rings, and $M$ an $A$-$B$-bimodule. Let \[ \Lambda= \begin{bmatrix}
A & _{A}M_{B} \\
0 & B
\end{bmatrix}\] be the triangular matrix ring. 
By \cite[Example 3.4]{AKLY}, there is a ladder of height 2
$$\xymatrix{
\D(A)
\ar@<-2.25ex>[rr]|-{i_{\#}}
\ar@<0.75ex>[rr]|-{\Lambda e_{2}\otimes^{L}-}
&&
\ar@<-2.25ex>[ll]
\ar@<0.75ex>[ll]|-{e_{2}\Lambda\otimes-}
\D(\Lambda)
\ar@<-2.25ex>[rr]|-{j^{\#}}
\ar@<0.75ex>[rr]
&&
\ar@<-2.25ex>[ll]|-{\Lambda e_{1}\otimes^{L}}
\ar@<0.75ex>[ll]|-{j_{*}}
\D(B)}$$
where $e_{1}=\begin{bmatrix}
0 & 0 \\
0 & 1
\end{bmatrix}$  and $e_{2}=\begin{bmatrix}
1 & 0 \\
0 & 0
\end{bmatrix}$.

Assume $\cpx{X}$ and $\cpx{Y}$ are cosilting complexes in $\D(A)$ and $\D(B)$ respectively. According to Theorem~\ref{thm:glue cosilting complex}, the glued cosilting object of $\cpx{X}$ and $\cpx{Y}$ {\color{black}along the upper-recollement} is a cosilting complex of $\D(\Lambda)$.
\end{Bsp}

\begin{Bsp}\label{example: cosilting complex} 
Let $A$ be a hereditary algebra over a field $k$ with the quiver
$$
\xymatrix{\circ\ar@<-0.4ex>[r]_(1){2}_(0){1}
&\circ\ar@<-0.4ex>[r]_(1){3}
&\circ
}.$$
Thanks to \cite[Example 3.4]{AKLY}, we have a ladder of height 2 
$$\xymatrix{
\D(B)
\ar@<-2.25ex>[rrr]
\ar@<0.75ex>[rrr]|-{i_{*}=B\otimes_B^{L}-}
&&&
\ar@<-2.25ex>[lll]
\ar@<0.75ex>[lll]|-{\RHom_{A}(B,-)}
\D(A)
\ar@<-2.25ex>[rrr]|-{j^{\#}}
\ar@<0.75ex>[rrr]
&&&
\ar@<-2.25ex>[lll]|-{Ae_{3}\otimes^{L}}
\ar@<0.75ex>[lll]
\D(k)}$$
where $Ae_{3}=P_{3}, j^{\#}=\RHom_{A}(I_3,-)$, and the algebra $B$ is isomorphic to a hereditary algebra with the quiver
$$
\xymatrix{\circ\ar@<-0.4ex>[r]_(1){2}_(0){1}
&\circ
}.$$ 

Consider the cosilting complex $k$ in $\D(k)$, which corresponds to the standard $t$-structure $(\D^{\leq 0}, \D^{\geq 0})$ with an adjacent co-$t$-structure $(\D^{\geq 0}, \D^{\leq 0})$. Now, select the cosilting complex $D(B)$, the dual $\Hom_k(B,k)$ of $B$, in $\D(B)$. Since $I_3=P_1$ is a projective $A$-module, $\RHom_{A}(I_3,-)=\Hom_{A}(I_3,-)$.We then have
$$j^{\#}i_{*}(D(B))=\Hom_{A}(I_3,D(B))\simeq k^2\in \D^{\leq 0}.$$
This implies that $V$ in the decomposition triangle of $i_*(D(B))$ is zero. And we have $U\simeq i_*(D(B))\simeq I_1\oplus I_2$.
Therefore, the glued cosilting complex in $\D(A)$ of the pair $(D(B), k)$ is
$$j_{*}(k)\oplus U\simeq I_{3}\oplus I_{1}\oplus I_2\simeq D(A).$$
\end{Bsp}

\section{Compatibility of mutation with gluing}\label{sec:mutation}
In this section, we will explore the compatibility of cosilting gluing and cosilting mutation, with a particular focus on pure-injective cosilting objects.
\subsection{Notation}
\begin{Def}
Let $\A$ be an abelian category. A {\em torsion pair} (\cite{D66}) in $\A$ is a pair $(\mathcal{U}, \mathcal{V})$ of full subcategories satisfying 
\begin{enumerate}
    \item ${\A}(U,V)=0$, for all $U\in\mathcal{U}$ and $V\in\mathcal{V}$;
    \item For each object $X$ of $\A$, there is an exact sequence
$$0\longrightarrow U_{X}\longrightarrow X\longrightarrow V_{X}\longrightarrow 0$$
where $U_{X}\in\mathcal{U}$ and $V_{X}\in \mathcal{V}$.
\end{enumerate}
\end{Def}

\begin{Def}
    Let $\A, \B$ and $\C$ be abelian categories. A {\em recollement} (\cite{FP}), denoted by a triple $(\A, \B,\C)$, is a diagram
   $$\xymatrix{\A\ar^-{i_*=i_!}[r]&\B\ar^-{j^!=j^*}[r]
\ar^-{i^!}@/^1.2pc/[l]\ar_-{i^*}@/_1.6pc/[l]
&\C\ar^-{j_*}@/^1.2pc/[l]\ar_-{j_!}@/_1.6pc/[l]}$$
of abelian categories and additive functors such that:
\begin{enumerate}
    \item $(i^*,i_*),(i_!,i^!),(j_!,j^!)$ and $(j^*,j_*)$ are adjoint pairs;
    \item $i_*,j_*$ and $j_!$ are fully faithful functors;
    \item $\mathsf{Im}i_*=\mathsf{Ker}j^*$.
\end{enumerate}
\end{Def}

Let $(\mathcal{U}_{1}, \mathcal{V}_{1})$ and $(\mathcal{U}_{2}, \mathcal{V}_{2})$ be torsion pairs in $\A$ and $\C$, respectively. Along the recollement above, the {\em glued torsion pair} (\cite[Theorem 1]{MH}) $(\mathcal{U}, \mathcal{V})$ in $\B$ is defined as follows:
$$\mathcal{U}:=\{B\in \B|i^{*}(B)\in \mathcal{U}_{1}, j^{*}(B)\in \mathcal{U}_{2}\},$$
$$\mathcal{V}:=\{B\in \B|i^{!}(B)\in \mathcal{V}_{1}, j^{*}(B)\in \mathcal{V}_{2}\}.$$

Let $\T$ be a triangulated category, and let $(\X,\Y)$ be a $t$-structure on $\T$. Denote by $\HH$ the heart of $(\X,\Y)$, assume $\HH$ has a torsion pair $(\mathcal{U},\mathcal{V})$. According to \cite{HRS} (or \cite[Lemma 1.1.2]{P} or \cite[Proposition 2.3]{ALSV22}), the {\em left HRS-tilt $t$-structure} of $(\X,\Y)$ at $(\mathcal{U},\mathcal{V})$ is given by
$$({^t\X},{^t\Y}):=(\X[1]*\mathcal{U},\mathcal{V}[1]*\Y)$$
with heart ${^t\HH}:=\mathcal{V}[1]*\mathcal{U}$. The {\em right HRS-tilt $t$-structure} of $(\X,\Y)$ at $(\mathcal{U},\mathcal{V})$ is given by
$$({\X^t},{\Y^t}):=(\X*\mathcal{U}[-1],\mathcal{V}*\Y[-1])$$
with heart ${\HH^t}:=\mathcal{V}*\mathcal{U}[-1]$.

Given a recollement of triangulated categories
$$\xymatrix{\mathscr{T}_1\ar^-{i_*=i_!}[r]
&\mathscr{T}\ar^-{j^!=j^*}[r]\ar^-{i^!}@/^1.2pc/[l]\ar_-{i^*}@/_1.6pc/[l]
&\mathscr{T}_2.\ar^-{j_*}@/^1.2pc/[l]\ar_-{j_!}@/_1.6pc/[l]}$$
Let $(\X_1,\Y_1)$ and $(\X_2,\Y_2)$ be $t$-structures of $\mathscr{T}_1$ and $\mathscr{T}_2$, respectively, and let $(\X,\Y)$ be the glued $t$-structure in $\T$. The recollement of triangulated categories induces a recollement of hearts (see \cite[Theorem 2.8]{PV}). Denote the hearts of the $t$-structures $(\X_1,\Y_1),~(\X_2,\Y_2)$ and $(\X,\Y)$ by $\HH_1, \HH_2$ and $\HH$, respectively. Let $\H_1^{0}, \H_2^{0}$ and $\H^{0}$ be the cohomological functors of the corresponding $t$-structures, and let $\epsilon_1: \HH_1\to \mathscr{T}_1, \epsilon_2: \HH_2\to \mathscr{T}_2$ and $\epsilon: \HH\to \T$ be the canonical embeddings. Then there is a recollement of abelian categories:
$$\xymatrix{\HH_1\ar^-{I_*=I_!}[r]&\HH\ar^-{J^!=J^*}[r]
\ar^-{I^!}@/^1.2pc/[l]\ar_-{I^*}@/_1.6pc/[l]
&\HH_2\ar^-{J_*}@/^1.2pc/[l]\ar_-{J_!}@/_1.6pc/[l]}$$
where the functors are defined as follows:
$$I^*=\H_1^0\circ i^{*}\circ \epsilon,~J_!=\H^0\circ j_{!}\circ \epsilon_2,$$
$$I_*=\H^0\circ i_{*}\circ \epsilon_1,~J^*=\H_2^0\circ j^{*}\circ \epsilon,$$
$$I^!=\H_1^0\circ i^{!}\circ \epsilon,~J_*=\H^0\circ j_{*}\circ \epsilon_2.$$

\subsection{HRS-tilts and recollements}
Now, let $(\mathcal{U}_{1}, \mathcal{V}_{1})$ and $(\mathcal{U}_{2}, \mathcal{V}_{2})$ be torsion pairs in $\HH_1$ and $\HH_2$, respectively, and let $(\mathcal{U}, \mathcal{V})$ be the glued torsion pair in $\HH$. Let $(\X_1^t,\Y_1^t)$ and $(\X_2^t,\X_2^t)$ be the right HRS-tilts of $(\X_1,\Y_1)$ and $(\X_2,\Y_2)$ at $(\mathcal{U}_{1}, \mathcal{V}_{1})$ and $(\mathcal{U}_{2}, \mathcal{V}_{2})$, respectively. 
\begin{Lem}\label{lem:HRS reco}\textnormal{(\cite[Theorem 6.4, Proposition 6.5]{LVY})}
   The right HRS-tilt $(\X^t, \Y^t)$ of the glued $t$-structure $(\X,\Y)$ from $(\X_1,\Y_1)$ and $(\X_2,\Y_2)$ at $(\mathcal{U}, \mathcal{V})$ coincides with the glued $t$-structure from $(\X_1^t,\Y_1^t)$ and $(\X_2^t,\X_2^t)$. 
\end{Lem}

We visualize the lemma above in the following ``commutative" diagram:
$$\xymatrix{
(\X_1, \Y_1)\ar[rr]^{\text{glue}}\ar[dd]|-{(\mathcal{U}_1, \mathcal{V}_1)}
&
&(\X, \Y)\ar[dd]|-{(\mathcal{U}, \mathcal{V})}
&
&(\X_2, \Y_2)\ar[ll]_{\text{glue}}\ar[dd]|-{(\mathcal{U}_2, \mathcal{V}_2)}
\\
 \quad\quad \quad\ar[rr]^{{\text{glue}}}& &\quad\quad & &\quad\quad\quad\ar[ll]_{{\text{glue}}}
\\
(\X_1^t, \Y_1^t)\ar[rr]^{\text{glue}}
&
&(\X^t, \Y^t)
&
&(\X_2^t, \Y_2^t)\ar[ll]_{\text{glue}}
}$$

Let $\T$ be a triangulated category and $(\X,\Y)$ a $t$-structure with heart $\HH$. A torsion pair $(\mathcal{U},\mathcal{V})$ in $\HH$ is said to be a {\em cosilting torsion pair} (\cite{ALSV22}) if the right HRS-tilt of $(\X,\Y)$ at $(\mathcal{U},\mathcal{V})$ is $(^{\perp_{<0}}C, {^{\perp_{>0}}C})$ for some cosilting object $C$ in $\T$.

\begin{Koro}\label{cor:glue cosilting torsion pair}
{\color{black}Let the diagram
$$\xymatrix{\mathscr{T}_1\ar^-{i_*=i_!}[r]
&\mathscr{T}\ar^-{j^!=j^*}[r]\ar^-{i^!}@/^1.2pc/[l]\ar_-{i^*}@/_1.6pc/[l]
&\mathscr{T}_2.\ar^-{j_*}@/^1.2pc/[l]\ar_-{j_!}@/_1.6pc/[l]}$$
be a recollement of triangulated categories. Let $C_1$ and $C_2$ be cosilting objects of $\T_1$ and $\T_2$ with corresponding $t$-structures $(\X_1, \Y_1)$ and $(\X_2, \Y_2)$, respectively. Suppose that $(\mathcal{U}_1,\mathcal{V}_1)$ and $(\mathcal{U}_2,\mathcal{V}_2)$ are cosilting torsion pairs with respect to the $t$-structures $(\X_1, \Y_1)$ and $(\X_2, \Y_2)$ with associated cosilting objects $C'_1$ and $C'_2$, respectively. If $(C_1, C_2)$ and $(C'_1, C'_2)$ are gluable pairs, then the glued torsion pair $(\mathcal{U},\mathcal{V})$ of $(\mathcal{U}_1,\mathcal{V}_1)$ and $(\mathcal{U}_2,\mathcal{V}_2)$ is a cosilting torsion pair.}
\end{Koro}
\begin{proof}
It follows directly from Lemma~\ref{lem:HRS reco} and Theorem~\ref{thm:glue cosilting}.
\end{proof}

\subsection{Mutation of cosilting objects}
A method of mutating cosilting objects is introduced by Angeleri H\"ugel-Laking-Stovick-Vit\'oria in \cite[Section 3]{ALSV22}. Since the definitions of left and right mutations are dual, we will focus only on right mutations. Let $C$ be a cosilting object and $\X$ be a subcategory of $\T$. Denote $\H_C$ as the heart and 
$\H^0_C$ as the cohomological functor of the corresponding $t$-structure of the cosilting object $C$. {\color{black}Furthermore, we define the notation
$$\H^0_C(\X):=\{\H^0_C(X)\mid X\in\X\}$$
and ${\rm Cogen}(\H^0_{C}(\mathcal{E}))$ is the subcategory formed by all subobjects of products of objects in $\H^0_{C}(\mathcal{E})$.}
\begin{Def}\cite[Definition 3.2, Proposition 3.10]{ALSV22}
    Let $\T$ be a triangulated category with products, and let $C$ be a cosilting object. Let $\mathcal{E}=\Prod(\mathcal{E})\subset \Prod(C)$. $C$ admits a right mutation $C'$ with respect to $\mathcal{E}$ if
    such that
    \begin{enumerate}
        \item $C$ admits an $\mathcal{E}$-precover, and 
        \item the torsion pair $(^{\perp_{0}}(\H^0_{C}(\mathcal{E})), {\rm Cogen}(\H^0_{C}(\mathcal{E})))$, {\color{black}cogenerated by $\H^0_{C}(\mathcal{E})$}, is a cosilting torsion pair.
    \end{enumerate}
\end{Def}

In this paper, we only consider the mutation of pure-injective cosilting objects, and in this case, we have the following useful lemma.

\begin{Prop}\label{mutation pair}\textnormal{(\cite[Lemma 4.8, Theorem 4.9]{ALSV22})}
Let  $C$ be a pure-injective cosilting object and  $\mathcal{E}=\Prod(\mathcal{E})\subset \Prod(C)$. Let $(\mathcal{S},\mathcal{R})$ be the torsion pair cogenerated by $\H^0_{C}(\mathcal{E})$. Then $C$ admits a right mutation $C'$ with respect to $\mathcal{E}$ if and only if $(\mathcal{S},\mathcal{R})$ is a cosilting torsion pair. Moreover, in this case, $C'$ is also pure-injective.
\end{Prop}

The following lemma is used to prove the main theorems.
\begin{Lem}\label{lem:useful}
Let $\mathscr{T}$ be a triangulated category.
\begin{enumerate}
    \item \textnormal{(\cite[Lemma~2.1]{ALSV22})}~Let $(\X, \Y)$ be a $t$-structure in $\mathscr{T}$ with heart $\mathcal{H}$. Denote $\H^0$ the associated cohomological functor of $\mathcal{H}$. 
    If $X\in \mathcal{H}$ and $Y\in \Y$, then $\mathscr{T}(X, Y)\simeq \mathcal{H}(X, \H^0(Y)).$
    \item \textnormal{(\cite[Lemma~2.5]{ALSV22})} Assume $\mathscr{T}$ has products and $C$ is a cosilting object. Then the associated heart $\mathcal{H}$ of $C$ is abelian with enough injective objects, and the functor $\H^0$ induces an equivalence from $\Prod(C)$ to $\mathsf{Inj}(\mathcal{H})$ and a natural isomorphism of functors
    $\mathscr{T}(-, C)\simeq \mathcal{H}(\H^0(-), \H^0(C)).$ 
\item {\color{black}Furthermore, if $X\in \Prod(C)$, then there is a natural isomorphism of functors
    $\mathscr{T}(-, X)\simeq \mathcal{H}(\H^0(-), \H^0(X)).$}
\end{enumerate}
\end{Lem}
{\color{black}
\begin{proof}
We prove statement (3). Denote by $(\X,\Y)=(^{\perp_{<0}}C, {^{\perp_{>0}}C})$ the associated $t$-structure of $C$, according to \cite[Lemma 3.1(2)]{PS15}, the functor $\H^0:\Y\to \mathcal{H}$ admits a left adjoint, which is precisely the embedding functor $j:\mathcal{H}\to \Y$, thus $\H^0$ respects products in $\Y$. Since $\Prod(C)$ is a subcategory of $\Y$, it follows that $\H^0$ also respects products in $\Prod(C)$. This implies the statement.
\end{proof}}

\subsection{Compatibility of mutation and gluing}
{\color{black}Let us fix the setting and notation for this subsection.

{\bf SETTING:}\quad 
Let $\T_1,\T_2$ and $\T_3$ be algebraically compactly generated triangulated categories. that form a ladder of height 2:
$$\xymatrix{
\mathscr{T}_1
\ar@<-2.25ex>[rr]|-{i_{\#}}
\ar@<0.75ex>[rr]|-{i_{*}}
&&
\ar@<-2.25ex>[ll]|-{i^{*}}
\ar@<0.75ex>[ll]|-{i^{!}}
\T
\ar@<-2.25ex>[rr]|-{j^{\#}}
\ar@<0.75ex>[rr]|-{j^{*}}
&&
\ar@<-2.25ex>[ll]|-{j_{!}}
\ar@<0.75ex>[ll]|-{j_{*}}
\mathscr{T}_2}.$$
Let $C_1$ and $C_2$ be pure-injective cosilting objects in $\T_1$ and $\T_2$, respectively. It follows from Theorem~\ref{thm:glue pure-inj cosilting} that $(C_1, C_2)$ is a gluable pair along the upper recollement. Denote the triangle (see Definition \ref{def:gluable cosilting}) by
$$V\longrightarrow i_{*}(C_1)\longrightarrow U\longrightarrow V[1]~\text{with}~V\in j_{*}(^{\perp_{\geq 0}}C_2), U\in (j_{*}(^{\perp_{\geq 0}}C_2))^{\perp}.$$
We denote the glued cosilting object by $C=j_{*}(C_2)\oplus U$.

{\bf NOTATION:} 
\begin{align*}
\mathcal{H}&:={ }^{\perp_{<0}}C\cap { }^{\perp_{>0}}C\\
\mathcal{H}_i&:={ }^{\perp_{<0}}C_i\cap { }^{\perp_{>0}}C_i,~i=1,2 \\
\H^0&: \T\to \mathcal{H}\\
\H^0_i&: \T_i\to \mathcal{H}_i,~i=1,2 
\end{align*}}

\subsubsection{Mutation at left side}

In this subsection, we will prove the following main result.

\begin{Theo}\label{thm:muta at left}
Let $X$ be a direct summand of $C_1$, and let $\mathcal{E}_1=\Prod(X)\subset \Prod(C_1)$. Assume $C_1$ admits a right mutation $C'_1$ with respect to $\mathcal{E}_1$. Denote by $C'$ the glued cosilting object of $C'_1$ and $C_2$. Then $C'$ is a right mutation of $C$.
\end{Theo}

To prove the theorem, let us begin with some lemmas.

\begin{Lem}\label{lem:for glue torsion pair}
 Let $X\in \mathcal{H}$. Then
 \begin{enumerate}
     \item $\T(X, j_*j^*(U)[-1])=0$. If $\T(X, U)=0$, then $\T(X, i_*i^!(U))=0$;
     \item If $\T(X, j_*(C_2))=0$, then $\T(X, j_*j^*(U))=0$;
     \item If $\T(X, j_*(C_2))=0$, then for $U_0\in \Prod(U)$,
     $$\T(X, i_*i^!(U_0))=0~\text{if and only if}~\T(X, U_0)=0.$$
 \end{enumerate}
\end{Lem}
\begin{proof}
$(1)$ Since $j^*(X)\in \X_2\cap\Y_2=\mathcal{H}_2$ and $j^*(U)\simeq j^*(V[1])\in \Y_2$, then $\T_2(j^*(X), j^*(U)[-1])=0$.

Note that $U$ has a canonical triangle
$$i_*i^!(U)\longrightarrow U\longrightarrow j_*j^*(U)\longrightarrow i_*i^!(U)[1].$$
Applying $\T(X, -)$ to this triangle gives us an exact sequence:
$$\T(X, j_*j^*(U)[-1])\longrightarrow \T(X, i_*i^!(U))\longrightarrow \T(X, U)\longrightarrow \T(X, j_*j^*(U)).$$
Hence, provided $\T(X, U)=0$, it follows that $\T(X, i_*i^!(U))=0$.

$(2)$ 
In the triangle $(*)$ (see Definition~\ref{def:gluable cosilting}), we set $V[1]\simeq j_*(\Bar{V})$ for some $\Bar{V}\in {^{\perp_{>0}}C_2}$. Let $\phi:\Bar{V}\to E_2$ be a $\Prod(C_2)$-preenvelope of $\Bar{V}$, and extend it to a triangle:
$$\Bar{V}\stackrel{\phi}\longrightarrow E_2\longrightarrow D\longrightarrow \Bar{V}[1].$$
Since $\Bar{V}, E_2\in {^{\perp_{>0}}C_2}$, {\color{black}applying the functor $\Hom_{\T}(-,C_2[i])$ to this triangle gives:
$$\T(E_2, C_2[i-1])\to \T(\Bar{V},C_2[i-1])\to \T(D,C_2[i])\to \T(E_2,C_2[i]).$$
For $i>1$, $\T(\Bar{V},C_2[i-1])\simeq \T(E_2,C_2[i])\simeq 0$, which implies $\T(D,C_2[i])\simeq 0$. When $i=1$, since $\phi:\Bar{V}\to E_2$ is a $\Prod(C_2)$-preenvelop of $\Bar{V}$, we know that $\T(E_2, C_2)\to \T(\Bar{V},C_2)$ is an epimorphism, leading us to conclude $D\in {{}^{\perp_{>0}}C_2}$.} By applying $j_*$ and $\H^0$ to the last triangle, there is an exact sequence
$$\H^{-1}(j_*(D))\longrightarrow\H^0(j_*(\Bar{V}))\stackrel{\H^0(j_*\phi)}\longrightarrow \H^0(j_*(E_2))\longrightarrow \H^0(j_*(D)).$$
Since $j_*({{}^{\perp_{>0}}C_2})\subseteq {{}^{\perp_{>0}}C}$, then $\H^{-1}(j_*(D))=0$. Moreover, 
\begin{align*}
    &\T(X, j_*(C_2))=0~~(\text{assumption})\\
    \Longrightarrow& \mathcal{H}(X, \H^0(j_*(C_2)))=0~~(\text{Lemma}~\ref{lem:useful}(1))\\
    \Longrightarrow& \mathcal{H}(X, \H^0(j_*(E_2)))=0~~({\color{black}\text{Lemma}~\ref{lem:useful}(3)~\text{and}~j_*~\text{respects products}})\\
    \Longrightarrow &\mathcal{H}(X, \H^0(j_*(\Bar{V})))=0\\
    \Longrightarrow &\mathcal{H}(X, \H^0(V[1]))=0\\
    \Longrightarrow &\T(X, V[1])=0~~(\text{Proposition}~\ref{rem:from condition tri}(2)~\text{and}~\text{Lemma}~\ref{lem:useful}(1))\\
    \Longrightarrow &\T(X, j_*j^*(U))=0.
\end{align*}

$(3)$
By $(1)$ and $(2)$, we have $\T(X, j_*j^*(U)[i])=0$ for $i=-1,0$. Applying the functor ${\T}(X,-)$ to the canonical triangle
$$i_*i^{!}(U)\longrightarrow U\longrightarrow j_*j^*(U)\longrightarrow i_*i^{!}(U)[1]$$
yields ${\T}(X,i_*i^{!}(U))\simeq {\T}(X,U)$. Thus (3) holds.
\end{proof}

To reach our final goal, the lemma below is crucial. Set

$$(\mathcal{S}_1, \mathcal{R}_1):=({ }^{\perp_0}\H^0_1(\mathcal{E}_1), \mathsf{Cogen}(\H^0_1(\mathcal{E}_1))).$$

\begin{Lem}\label{lem:main gluable pair}
$(\mathcal{S},\mathcal{R})$ is the glued torsion pair of 
$(\mathcal{S}_1,\mathcal{R}_1)$ and $(0, \mathcal{H}_2)$ along the recollement of hearts. 
\end{Lem}
\begin{proof}
Set $\mathcal{E}:=\Prod(j_*(C_2))\oplus \Prod(U_X)$. 
It suffices to prove that 
$${ }^{\perp_0}\H^0(\mathcal{E})=\{X\in\mathcal{H}~|~I^*(X)\in\mathcal{S}_1, J^*(X)=0\}.$$
Note that there are equivalences
\begin{align*}
   X\in { }^{\perp_0}\H^0(\mathcal{E})\Leftrightarrow & \H(X, \mathcal{H}^0(\mathcal{E}))=0\\
   \Leftrightarrow &\T(X, \mathcal{E})=0 ~({\rm Lemma}~\ref{lem:useful}(1))\\
   \Leftrightarrow & \T(X, j_*(C_2))=0~\text{and}~\T(X, U_X)=0
\end{align*}
and
\begin{align*}
I^*(X)\in \mathcal{S}_1\Leftrightarrow &\H^0_1\circ i^*\circ \epsilon(X)\in \mathcal{S}_1\\
\Leftrightarrow& \H^0_1\circ i^*(X)\in \mathcal{S}_1\\
\Leftrightarrow&\mathcal{H}_1(\H^0_1(i^*(X)),\H^0_1(\mathcal{E}_1))=0\\
\Leftrightarrow&\T_1(i^*(X),\mathcal{E}_1)=0~~({\color{black}\text{Lemma}~\ref{lem:useful}(3)})\\
\Leftrightarrow&\T(X, i_*(\mathcal{E}_1))=0  \\  
\Leftrightarrow& \T(X, i_*i^!(\Prod(U_X))=0
\end{align*}
and 
\begin{align*}
    J^*(X)=0\Leftrightarrow &\H^0_2\circ j^*\circ\epsilon(X)=0\\
    \Leftrightarrow & \H^0_2(j^*(X))=0\\
    \Leftrightarrow &\mathcal{H}_2(\H^0_2(j^*(X)), \H^0_2(C_2))=0\\
    \Leftrightarrow &\T_2(j^*(X), C_2)=0~~(\text{Lemma}~\ref{lem:useful}(2))\\
    \Leftrightarrow &\T(X, j_*(C_2))=0.
\end{align*}
Thanks to Lemma~\ref{lem:for glue torsion pair}(3), we know that $X\in { }^{\perp_0}\H^0(\mathcal{E})$ if and only if $I^*(X)\in\mathcal{S}_1$ and $J^*(X)=0$. This finishes the proof.
\end{proof}

\begin{proof}[The proof of Theorem~\ref{thm:muta at left}]
From Proposition~\ref{mutation pair} and Theorem~\ref{thm:glue pure-inj cosilting}, $C'_1$ is pure-injective and the pair $(C'_1, C_2)$ is gluable. Thus $C'$ is well-defined and is a pure-injective cosilting object.

Since $X$ is a direct summand of $C_1$, we can assume $X\oplus Y\simeq C_1$. By \cite[Theorem 6.13]{A19}, $(^{\perp_{<0}}C_2,^{\perp_{>0}}C_2)$ has an adjacent co-$t$-structure. By the proof of Proposition \ref{prop:sufficient glue}, $X$ and $Y$ have the following triangles:
$$V_X\longrightarrow i_{*}(X)\longrightarrow U_X\longrightarrow V_X[1]$$
$$V_Y\longrightarrow i_{*}(Y)\longrightarrow U_Y\longrightarrow V_Y[1]$$
where $V_X, V_Y$ are in $j_{*}(^{\perp_{\geq 0}}C_2)$ and  $U_X, U_Y$ are in $(j_{*}(^{\perp_{\geq 0}}C_2))^{\perp}$. Then $C_1$ has the following decomposition triangle:
$$V\longrightarrow i_{*}(C_1)\longrightarrow U\longrightarrow V[1]$$
where $V\simeq V_X\oplus V_Y$ and $U\simeq U_X\oplus U_Y$, which satisfies $(*)$ in Definition \ref{def:gluable cosilting}.

We claim  $C'$ is a right mutation of $C$ with respect to $\mathcal{E}=\Prod(j_*(C_1)\oplus U)$. For this reason, let $(\mathcal{S}, \mathcal{R}):=({ }^{\perp_0}\H^0(\mathcal{E}), \mathsf{Cogen}(\H^0(\mathcal{E})))$.

By Theorem \ref{thm:glue pure-inj cosilting}, $C$ is a pure-injective cosilting object. \textcolor{black}{By Lemma \ref{lem:HRS reco} and Lemma \ref{lem:main gluable pair} below, the right HRS-tilt of $(^{\perp_{<0}}C, ^{\perp_{>0}}C)$ with respect to $(\mathcal{S}, \mathcal{R})$ is just $(^{\perp_{<0}}C', ^{\perp_{>0}}C')$. Hence $(\mathcal{S}, \mathcal{R})$ is a cosilting torsion pair. By Lemma \ref{mutation pair}, we have that $C'$ is a right mutation of $C$ with respect to $\mathcal{E}$.}
\end{proof}

\begin{Bsp}\label{example:right mutation}
Following Example~\ref{example: cosilting complex},  
we know that the injective module $D(B)$, as a stalk complex, is a cosilting complex in $\D(B)$. And $D(B)$ admits a right mutation at $\Prod(I_2)$. Indeed, there exists a short exact sequence
$$0\longrightarrow P_2\longrightarrow I_2\stackrel{\phi}{\longrightarrow} I_1\to 0,$$
where $\phi$ is a $\Prod(I_2)$-precover of $I_1$. A right mutation of $D(B)$ at $\Prod(I_2)$ is the stalk complex of the regular module $B=P_2\oplus I_2$. Moreover, there is a triangle
$$I_1[-1]\longrightarrow P_2 \stackrel{\varphi}{\longrightarrow} P_1\longrightarrow I_1$$
where $\varphi$ is a $\Prod(P_2)$-precover of $P_1$. A right mutation of $B=P_2\oplus P_1$ at $\Prod(P_2)$ is the complex $I_1[-1]\oplus P_2$.

Analogous to the discussion in Example~\ref{example: cosilting complex}, the glued cosilting complexes of the pairs $(D(B),k)$ and $(B,k)$ are $D(A)$ and $S_2\oplus I_2\oplus I_3$, respectively. According to Theorem \ref{thm:muta at left}, $S_2\oplus I_2\oplus I_3$ is a right mutation of $D(A)$ at $\Prod(I_2\oplus I_3)$.
\end{Bsp}

\subsubsection{Mutation at right side}
When considering the right side, the situation becomes intricate. Indeed, for compatibility to be possible, the mutation needs to satisfy a special assumption based on our proof.

\begin{Theo}\label{thm:muta at right}
Let $\mathcal{E}_2=\Prod(\mathcal{E}_2)$ be a subcategory of $\Prod(C_2)$. Assume $C_2$ admits a right mutation $C'_2$ with respect to $\mathcal{E}_2$ and such that  \textcolor{black}{$V[1]\in j_*(^{\perp_{> 0}}C'_2)$}. Denote by $C'$ the glued cosilting object of $C_1$ and $C'_2$. Then $C'$ is a right mutation of $C$.
\end{Theo}

{\color{black}For convenience, we fixed the notation under the assumption of the theorem above:
\begin{align*}
(\mathcal{S}_2, \mathcal{R}_2)&:=({ }^{\perp_0}\H^0_2(\mathcal{E}_2), \mathsf{Cogen}(\H^0_2(\mathcal{E}_2)))\\
(\mathcal{S}_r, \mathcal{R}_r)&:~\text{the glued pair of}~(0, \mathcal{H}_1)~\text{and}~(\mathcal{S}_2, \mathcal{R}_2)
\end{align*}}

\begin{proof}[The proof of Theorem~\ref{thm:muta at right}]

\textcolor{black}{Let $\mathcal{E}=\Prod(j_*(\mathcal{E}_2)\oplus U)$, we want to prove $C'$ is a right mutation of $C$ with respect to $\mathcal{E}$ under the given hypothesis. Let $(\mathcal{S}, \mathcal{R}):=({ }^{\perp_0}\H^0(\mathcal{E}), \mathsf{Cogen}(\H^0(\mathcal{E})))$. Since $C'_2$ is a right mutation of $C_2$ with respect to $\mathcal{E}_2$, we have $^{\perp_{> 0}}C'_2=\mathsf{Cogen}(\mathcal{E}_2)*{^{\perp_{>0}}C_2}[-1]\subseteq {^{\perp_{>0}}C_2}$ by definition.}



We claim that $\mathcal{S}=\mathcal{S}_r$. \textcolor{black}{In particular the torsion pair $(\mathcal{S},\mathcal{R})$ is a cosilting torsion pair.}

Let $X\in\mathcal{H}$. By a proof similar to that of Lemma~\ref{lem:main gluable pair}, we have:
    \begin{align*}
        I^*(X)=0 &~\text{if and only if} ~\T(X, i_*(C_1))=0, \\
        J^*(X)\in \mathcal{S}_2 &~\text{if and only if}~ \T(X,j_*(\mathcal{E}_2))=0,
    \end{align*}
and 
\begin{align*}
X\in { }^{\perp_0}\H^0(\mathcal{E})&\Leftrightarrow \T(X, \mathcal{E})=0 \\
&\Leftrightarrow \T(X, j_*(\mathcal{E}_2))=0~\text{and}~\T(X, U)=0.
\end{align*}

\textcolor{black}{~~($\mathcal{S}\subseteq \mathcal{S}_r$)~~}Assume $X\in\mathcal{S}$, that is, $\T(X, \mathcal{E})=0$. Then $\T(X, j_*(\mathcal{E}_2))=0~\text{and}~\T(X, U)=0$. By 
applying the functor $\T(X,-)$ to the triangle $(*)$(see Definition~\ref{def:gluable cosilting}) 
$$V\longrightarrow i_{*}(C_1)\longrightarrow U\longrightarrow V[1],~V\in j_{*}(^{\perp_{\geq 0}}C_2),~U\in (j_{*}(^{\perp_{\geq 0}}C_2))^{\perp}$$
we obtain an exact sequence
$$\T(X, V)\longrightarrow \T(X, i_*(C_1))\longrightarrow \T(X, U)\longrightarrow \T(X, V[1]).$$
Since $X\in \mathcal{H}$, we have $j^*(X)\in ^{\perp_{< 0}}C_2$, so the first item is zero, hence $\T(X, i_*(C_1))=0$. Thus we find $I^*(X)=0, J^*(X)\in \mathcal{S}_2$,  and consequently, $X\in \mathcal{S}_r$.

\textcolor{black}{~~($\mathcal{S}_r\subseteq \mathcal{S}$)~~}For the reverse inclusion, let $X\in\mathcal{S}_r$, we need to show $X\in \mathcal{S}$. \textcolor{black}{By the hypothesis}, $V[1]$ is also in $j_*(^{\perp_{> 0}}C'_2)$, so we can express $V[1]=j_*(V')$ for some $V'\in {^{\perp_{>0}}C'_2}$. 
Since $^{\perp_{>0}}C'_2= \mathsf{Cogen}(H_2^0(\mathcal{E}_2))*{^{\perp_{>0}}}C_2[-1]$, $V'$ admits a triangle:
$$W\longrightarrow V'\longrightarrow Z\longrightarrow W[1]$$
with $W\in \mathsf{Cogen}(H_2^0(\mathcal{E}_2))$ and $Z\in {^{\perp_{>0}}}C_2[-1]$.
The triangle above induces an exact sequence
$$\H^0j_*(Z[-1])\to\H^0j_*(W)\longrightarrow \H^0j_*(V')\longrightarrow \H^0j_*(Z)\longrightarrow \H^0j_*(W[1])$$
where $\H^0j_*(W)\in J_*(\mathsf{Cogen}(\H_2^0(\mathcal{E}_2)))$ and $\H^0j_*(Z)=0=\H^0j_*(Z[-1])$. Hence $\H^{0}(V[1])=\H^{0}(j_*(V'))\simeq \H^0j_*(W)\in J_*(\mathsf{Cogen}(\H_2^0(\mathcal{E}_2)))$. Thus, we have 
$$\T(X, V[1])\simeq \HH(X, \H^0(V[1]))\simeq \HH_2(J^*(X), W)=0$$
where the first isomorphism is from Lemma~\ref{lem:useful}(1), and the second is determined by the adjoint pair $(J^*, J_*)$, and the third follows because $J^*(X)\in \mathcal{S}_2$ and $W\in \mathcal{R}_2$.

By applying the functor ${\rm Hom}_{\T}(X,-)$ to the triangle $(*)$ in Definition \ref{def:gluable cosilting}, we find ${\rm Hom}_{\T}(X,U)=0$. Hence, $X\in \mathcal{S}$. 
Therefore, by Lemma \ref{lem:HRS reco}, we have that $(\mathcal{S},\mathcal{R})$ is a cosilting torsion pair.

By Theorem \ref{thm:glue pure-inj cosilting}, $C$ is a pure-injective cosilting object. Thus, we have shown that $(\mathcal{S},\mathcal{R})$ is a cosilting torsion pair. By Proposition \ref{mutation pair},  $C'$ is a right mutation of $C$.
\end{proof}

When comparing Theorem~\ref{thm:muta at right} with Theorem~\ref{thm:muta at left}, there is an additional condition ``$V[1]\in j_*(^{\perp_{> 0}}C'_2)$". This condition ensures that the gluable pairs $(C_1, C_2)$ and $(C_1,C'_2)$ share the same triangle (see Definition~\ref{def:gluable cosilting}), which significantly contributes to the proof.

\begin{Bsp}
Using the notation in Example~\ref{example: cosilting complex}, let $e=e_1+e_2$ and define  $C=A/AeA$. Then, there is a ladder 
    $$\xymatrix{
\D(C)
\ar@<-2.25ex>[rrr]
\ar@<0.75ex>[rrr]|-{i_{*}=C\otimes_B^{L}-}
&&&
\ar@<-2.25ex>[lll]
\ar@<0.75ex>[lll]|-{\RHom_{A}(C,-)}
\D(A)
\ar@<-2.25ex>[rrr]|-{j^{\#}}
\ar@<0.75ex>[rrr]
&&&
\ar@<-2.25ex>[lll]|-{Ae\otimes^{L}}
\ar@<0.75ex>[lll]|-{j_*=\RHom(eA,-)}
\D(B)}$$

The stalk complexes $C$ and $D(B)$ are cosilting objects in $\D(C)$ and $\D(B)$, respectively. Note that $i_*(C)\simeq S_3=P_3$ fits into the triangle
$$I_2[-1]\longrightarrow S_3\longrightarrow P_1\longrightarrow I_2$$
where $I_2[-1]\in j_*(^{\perp_{\geq 0}}D(B))$ and $P_1\in (j_*(^{\perp_{\geq 0}}D(B)))^{\perp}$. Consequently, the pair $(C, D(B))$ of cosilting objects is gluable along the upper recollement (the first three rows of the ladder). And the glued cosilting object is $j_*(D(B))\oplus P_1\simeq D(A)$ (see Theorem~\ref{thm:glue cosilting}). 

Additionally, we have $I_2[-1]\simeq j_*(_BI_2)[-1]\in j_*(^{\perp_{\geq 0}}B)$ and $P_1\in (j_*(^{\perp_{\geq 0}}B))^{\perp}$. With respect to the triangle mentioned above, the pair $(C, B)$ of cosilting objects is also gluable along the upper recollement. The resulting glued cosilting object is $I_3\oplus I_2\oplus S_2$.

Since $B$ is a right mutation of $D(B)$ (see Example~\ref{example:right mutation}) and 
$I_2\in j_*(^{\perp_{> 0}}B)$, the gluable pairs $(C, D(B))$ and $(C, B)$ satisfy the conditions in Theorem~\ref{thm:muta at right}. Therefore, the glued cosilting object $I_3\oplus I_2\oplus S_2$ associated with $(C, B)$ is a right mutation of the glued cosilting object $D(A)$ corresponding to $(C, D(B))$ in $\D(A)$. This is consistent with the conclusion in Example~\ref{example:right mutation}.
\end{Bsp}

\section{Declarations}

{\bf Funding} $\quad$ Yaohua Zhang is supported by the National Natural Science Foundation of China (No. 12401044).~\\

{\bf Availability of data and materials} $\quad$ All data generated or analyzed during this study are included in this published
article.~\\

{\bf Competing interests} $\quad$ The authors declare no competing interests.


\end{document}